\documentclass[12pt, reqno]{amsart}

\newtheorem{theorem}{Theorem}[section]
\newtheorem{lemma}[theorem]{Lemma}
\newtheorem{proposition}[theorem]{Proposition}
   
\newtheorem{definition}[theorem]{Definition}
\newtheorem{example}[theorem]{Example}

\newtheorem{remark}[theorem]{Remark}

\numberwithin{equation}{section}

\usepackage{times}
\usepackage{enumerate}
\usepackage{mathrsfs}
\usepackage{tfrupee}
\usepackage{tikz}
\usetikzlibrary{chains,fit}
\usetikzlibrary{shapes,snakes}
\usetikzlibrary{graphs}
\usepackage{graphicx,adjustbox}
\usepackage{hyperref}
\usepackage{amsmath,amssymb}
\usepackage{amscd}
\usepackage{graphicx}
\usepackage[all]{xy}

\usepackage{float}
\usepackage{caption}
\usepackage{subcaption}

\usepackage{cleveref}

\usepackage{geometry}
\usepackage[backend=bibtex]{biblatex}
\begin{document}

\title{On the Depth of Generalized Binomial Edge Ideals}
\author{
Anuvinda J, Ranjana Mehta, \and Kamalesh Saha 
}
\date{}

%Author1
\address{\small \rm Department of Mathematics,
SRM University - AP, Amaravati, Andhra Pradesh 522502, India.}
\email{anuvinda\_j@srmap.edu.in} 

%Author2
\address{\small \rm Department of Mathematics,
SRM University - AP, Amaravati, Andhra Pradesh 522502, India.}
\email{ranjana.m@srmap.edu.in}

%Author3
\address{\small \rm Chennai Mathematical Institute, Siruseri, Chennai, Tamil Nadu 603103, India}
\email{ksaha@cmi.ac.in; kamalesh.saha44@gmail.com}
%\thanks{Kamalesh Saha is funded by NBHM Post-Doctoral Fellowship, sponsored by the National Board of Higher Mathematics, Government of India.}

\date{}

\subjclass[2020]{Primary 13C15, 05E40, 13F65, 05C69}

\keywords{Generalized binomial edge ideals, depth, vertex-connectivity, $d$-compatible maps, accessible graphs}

\allowdisplaybreaks

\begin{abstract}
This research focuses on analyzing the depth of generalized binomial edge ideals. We extend the notion of $d$-compatible map for the pairs of a complete graph and an arbitrary graph, and using it, we give a combinatorial lower bound for the depth of generalized binomial edge ideals. Subsequently, we determine an upper bound for the depth of generalized binomial edge ideals in terms of the vertex-connectivity of graphs. We demonstrate that the difference between the upper and lower bounds can be arbitrarily large, even in cases when one of the bounds is sharp. In addition, we calculate the depth of generalized binomial edge ideals of certain classes of graphs, including cyclic graphs and graphs with Cohen-Macaulay binomial edge ideals.
\end{abstract}
\maketitle

\section{Introduction} 
Let $ m,n $ be two positive integers and $K $ be a field and $S=K[x_{ij}: i \in [m]\  and\  j \in [n]]$ be a polynomial ring in $mn$ variables, where $ [m]=\{1,2,..,m\} $ and $ [n]=\{1,2,..n\} $. The determinantal ideals are the ideals generated by all the $ t $-minors of a generic matrix $X=(x_{ij})$ and represented as $I_{t}(X)$. The determinantal ideals have been extensively studied due to their significance in representation theory, combinatorics, and invariant theory (see \cite{bc03}). Subsequently, people started studying the ideals generated by some arbitrary set of minors of a generic matrix owing to their prevalence in various fields.
\medskip

 The binomial edge ideals of graphs are an extension of the determinantal ideals $I_{2}(X)$, which are associated with a $2\times n$ generic matrix $X$. The introduction of this class of ideals was done independently in two separate works, namely \cite{hhhrkara} and \cite{ohtani}. Subsequently, in \cite{rauh13}, J. Rauh introduced the generalized binomial edge ideals, which represent a generalization of binomial edge ideals. The binomial edge ideals and generalized binomial edge ideals are important in algebraic statistics, particularly in examining conditional independence ideals. In \cite{ehhtq14}, Ene et al. introduced the binomial edge ideal of a pair of graphs, which serves as an extension of generalized binomial edge ideals. Significant advancements have been achieved so far in the study of binomial edge ideals. However, the algebraic invariants related to generalized binomial edge ideals and binomial edge ideals of pairs of graphs are relatively unexplored.
 \medskip
 
  Let $ G_{1} $ and $ G_{2} $ be two simple graphs on $ [m] $ and $ [n] $ respectively. Suppose $ e=\{i,j\} \in E(G_{1}) $ and $ f=\{k,l\} \in E(G_{2}) $ with $ i<j $ and $ k<l $. Assign a 2-minor $ p_{(e,f)}=[i\ j | k\ l]=x_{ik}x_{jl}-x_{il}x_{jk} $ to the pair $ (e,f) $. The binomial edge ideal of the pair $ (G_{1},G_{2}) $ is defined to be $$\mathcal{J}_{G_{1},G_{2}}=(p_{(e,f)} | e \in E(G_{1}), f \in E(G_{2}) ) $$ in $ S $. Let $ K_{m} $ denote the complete graph on $ m $ vertices. Then $\mathcal{J}_{K_{m},G} $ is the generalized binomial edge ideal associated with $ G $. It should be noted that in the case of $m=2$, $\mathcal{J}_{K_{2},G} =J_{G}$ is the classical `binomial edge ideal' of $ G $. In comparison to the binomial edge ideals of pairs of graphs, the generalized binomial edge ideals exhibit more favourable algebraic properties (see \cite{ehhtq14}, \cite{saeedi2013binomial}).
  \medskip
  
Depth is one of the most significant homological invariants associated with a module. In contrast to other homological invariants like Castelnuovo-Mumford regularity, the depth of the quotient rings by binomial edge ideals is less explored. For a finitely generated graded $S$-module $M$, the depth is defined as follows:
$$\mathrm{depth}(M)=\mathrm{min}\{i \mid H_{\mathfrak{m}}^{i}(M) \neq 0\},$$ 
where $H_{m}^{i}(M)$ denotes the $i^{th}$ local cohomology module of $M$ with respect to the homogeneous maximal ideal $ \mathfrak{m}=(x_{11},...,x_{mn})$ of $S$.
\medskip

The depth of generalized binomial edge ideals is only known for a very restricted set of classes (see to \cite{ci20}, \cite{ehhtq14}, \cite{zhu23partite}).  Even, there is no specific combinatorial upper or lower bound for the depth of these ideals. Whereas for binomial edge ideals, there are still some exact formulas and bounds on the depth. For the binomial as well as generalized binomial edge ideals of complete graphs, the depth is well-known as they are determinantal ideals. In \cite{banlui}, the authors established a combinatorial upper bound for the depth of binomial edge ideals in terms of vertex-connectivity of the graph. In fact, they showed that if $G$ is a connected non-complete graph, then $\mathrm{depth}(S/J_{G})\leq 2+n-\kappa(G)$, where $\kappa(G)$ is the vertex-connectivity of the graph $G$. Again, In \cite{rouzbahani2022depth}, the notion of $d$-compatible map has been introduced, and a combinatorial lower bound of $\mathrm{depth}(S/J_{G})$ had been established using that map. Specifically, they proved that for any simple graph $G$, $\mathrm{depth}(S/J_{G})\geq f(G)+d(G)$, where $ f(G) $ is the number of free vertices of $G$ and $d(G)$ denotes the sum of diameters of connected components of $G$ and the number of isolated vertices of $G$. In this paper, we give a combinatorial upper bound of the depth of generalized binomial edge ideals, which is a generalization of the upper bound given in \cite{banlui} for the depth of binomial edge ideals. Also, we extend the notion of $d$-compatible map, and using it, we establish a combinatorial lower bound of $\mathrm{depth}(S/\mathcal{J}_{K_m,G})$, which generalizes the lower bound of $\mathrm{depth}(S/J_G)$ given in \cite{rouzbahani2022depth}. Moreover, using our bounds and other techniques, we explicitly calculate the depth of generalized binomial edge ideals (independent of the characteristic of the base field $K$) of several classes of graphs, including cyclic graphs and graphs with Cohen-Macaulay binomial edge ideals. Additionally, we study to what extent the depth of generalized binomial edge ideals can be smaller than the provided upper bound and to what extent it can be larger than the given lower bound. The paper is organised in the following manner.
\medskip

In Section \ref{sec2}, we discuss the necessary prerequisites for the need of subsequent sections. For a graph $G$, we give a combinatorial lower bound of $\mathrm{depth}(S/\mathcal{J}_{K_m,G})$ in Section \ref{sec3}. In \cite{rouzbahani2022depth}, the authors defined the notion of $d$-compatible map, which is a map from $\mathcal{G}$ to $\mathbb{N}_{0}$, where $\mathcal{G}$ denotes the set of all simple graphs and $\mathbb{N}_{0}$ denotes the set of non-negative integers. In this section, We first extend the concept of $d$-compatible map from $\mathcal{K\times G}$ to $\mathbb{N}_{0}$, where $\mathcal{K}$ denotes the set of all complete graphs. By showing the existence of a $d$-compatible map, we give a combinatorial lower bound of $\mathrm{depth}(S/\mathcal{J}_{K_m,G})$ as follows:

\begin{theorem}[{Theorem \ref{thmd-comp}, \ref{thmlb}}]
    Let $G$ be a graph and $\Psi$ be a $d$-compatible map. Then $ \mathrm{depth}(S/\mathcal{J}_{K_{m},G}) \geq \Psi(K_{m},G)$. Moreover, $\Psi(K_m,G)=(m-2)t+f(G)+d(G)$ is a $d$-compatible map, where $t$ denotes the number of connected components of $G$. Hence, $\mathrm{depth}(S/\mathcal{J}_{K_{m},G}) \geq (m-2)t+f(G)+d(G)$.
\end{theorem}

\noindent In Section \ref{sec4}, we give a combinatorial upper bound of $\mathrm{depth}(S/\mathcal{J}_{K_{m},G})$ in terms of the vertex-connectivity of the graph $G$, the number of vertices of $G$ and $m$. Using the lower bound of the cohomological dimension of generalized binomial edge ideals given in \cite{katsabekis2022arithmetical}, we first establish the upper bound of $\mathrm{depth}(S/\mathcal{J}_{K_{m},G})$ in characteristic $p$ set up, and then, using the characteristic $p$ reduction, we prove the upper bound when $K$ is a field of characteristic zero. In particular, we prove the following:

\begin{theorem}[{Theorem \ref{thmupperbound}}]
   Let $G$ be a non-complete connected graph on $n$ vertices. If the vertex-connectivity of $G$ is $\kappa(G)$, then 
$\mathrm{depth} (S/\mathcal{J}_{K_{m},G})\leq m+n-\kappa(G)$.
\end{theorem}

\noindent Section \ref{sec5} of this paper focuses on determining the precise formula of the depth of generalized binomial edge ideals of certain classes of graphs, as well as comparing the depth to their respective bounds. In Theorem \ref{thmub=lb}, we provide an infinite class of graphs for which the lower bound given in Theorem \ref{thmlb} is equal to the upper bound given in Theorem \ref{thmupperbound}, and thus, is equal to $\mathrm{depth} (S/\mathcal{J}_{K_{m},G})$. In Theorem \ref{thmdepthcycle}, we calculate the depth of generalized binomial edge ideals of a class of graphs constructed from cyclic graphs. The technique used in Theorem \ref{thmdepthcycle} to calculate the depth of generalized binomial edge ideals of cyclic graphs is new. As a remark (Remark \ref{remub=lb+k}), we get for every $m\geq 2$ and $k\in\mathbb{N}$, there exists a connected graph $G_k$ for which $\mathrm{depth} (S/\mathcal{J}_{K_{m},G_k})$ is equal to the upper bound and $k$ more than the lower bound. In Theorem \ref{thmd=lb=ub-k}, we show that for every $m\geq 2$ and $k\in\mathbb{N}$, there exists a connected graph $G_k$ for which $\mathrm{depth} (S/\mathcal{J}_{K_{m},G_k})$ is equal to the lower bound and $k$ less than the upper bound. In \cite{acc}, Bolognini et al. introduced the notion of accessible graphs and strongly unmixed binomial edge ideals to classify the Cohen-Macaulay property of binomial edge ideals combinatorially. They demonstrate that the property of $J_G$ being strongly unmixed entails that $J_G$ is Cohen-Macaulay, which in turn requires that $G$ is accessible. They have also put up a conjecture on the converse of this statement. The conjecture has been proved for many classes (\cite{acc}, \cite{lmrr23}, \cite{sswhisker22}). In that direction, we calculate the depth of generalized binomial edge ideals of those graphs whose binomial edge ideals are strongly unmixed (see Theorem \ref{thmgensu}). If the conjecture is true, they are precisely the class of accessible graphs. Finally, we provide an example (Example \ref{exlb<d<ub}) of a graph $G$ for which $\mathrm{depth} (S/\mathcal{J}_{K_{m},G})$ is strictly greater than the lower bound and strictly less than the upper bound.

\section{Preliminaries}\label{sec2}
Let $G$ be a simple graph on $ [n] $. Let $ V(G) $ and $ E(G) $ denote the vertex set and edge set of $ G $, respectively. 
Let $T \subseteq [n]$. We write $G-T$ to denote the induced subgraph of $G$ on the vertex set $V(G)\setminus T$. By $G-v$, we mean the induced subgraph $G-\{v\}$, where $v \in V(G)$. Let $c_{G}(T)$ denote the number of connected components of $G-T$. A vertex $v\in V(G)$ is called a \textit{cut vertex} of $G$ if the number of connected components of $G$ is less than that of $G-v$. If $v$ is a cut vertex of the induced subgraph $G-(T\setminus \{v\})$ for each $v \in T$, then $T$ is said to have the \textit{cut point property} or $T$ is said to be a \textit{cut set} of $G$. Let $\tilde{G}$ represent the complete graph on $V(G)$ and $G_{1},G_{2},\cdots,G_{c_{G}(T)}$ be the connected components of $G-T$. Now, for a positive integer $ m\geq2 $ and $T\subseteq [n]$, we consider the ideal 
$$P_{T}(K_{m},G)=\big( \{x_{ij}:(i,j)\in [m] \times T\},\mathcal{J}_{K_{m},\tilde{G_{1}}},\mathcal{J}_{K_{m},\tilde{G_{2}}},\ldots,\mathcal{J}_{K_{m},\tilde{G}_{c_{G}(T)}}\big)$$ 
of the polynomial ring $S=K[x_{ij}:i\in [m],j \in [n]]$. Then $P_{T}(K_{m},G)$ is a prime ideal containing $\mathcal{J}_{K_{m},G}$. By \cite[Theorem 3.7]{rauh13}, $\mathcal{J}_{K_{m},G}$ is a radical ideal and the minimal prime ideals of $\mathcal{J}_{K_{m},G}$ are precisely the prime ideals $P_{T}(K_{m},G)$ for $T \in \mathcal{C}(G)$, where $\mathcal{C}(G)$ is the set of all cut sets of $G$. That is, 
$$ \mathcal{J}_{K_{m},G} = \bigcap_{T \in \mathcal{C}(G)}P_{T}(K_{m},G).$$

A simple graph is said to be \textit{complete} if there is an edge between every pair of vertices. A complete graph on $n$ vertices is denoted by $K_n$. Let $G$ be a simple graph. The \textit{neighbourhood} of a vertex $v \in V(G)$, denoted by $\mathcal{N}_{G}(v)$, is the set of vertices that are adjacent to $v$, i.e. $\mathcal{N}_{G}(v):= \{u \in V(G) : \{u, v\} \in E(G)\}$. For a vertex $v$ of $G$, let us define the graph $G_{v}$ as follows: $V(G_{v})=V(G) $ and $E(G_{v})=E(G)\cup \{\{u,w\} : u, w \in \mathcal{N}_{G}(v)\} $. A vertex $v\in V (G)$ is said to be a \textit{free} vertex of $G$ if the induced subgraph of $G$ on the vertex set $\mathcal{N}_{G}(v)$ is a complete graph. Otherwise, we call it a \textit{non-free} vertex. An induced subgraph of $G$ is called a \textit{clique} if it is complete.

\begin{definition}{\rm
A \textit{cycle} of length $n$, denoted by $C_n$, is a connected graph on $n$ vertices such that $V(C_n)=\{1,\ldots,n\}$ and $E(C_n)=\{\{i,i+1\}\mid 1\leq i\leq n-1\}\cup\{\{1,n\}\}$. A graph $G$ is called a \textit{chordal graph} if $G$ has no induced cycle of length $>3$.
}
\end{definition}

\begin{definition}{\rm
A graph $G$ is said to be a \textit{generalized block graph} if $G$ is chordal and for any three different maximal cliques $F_{i}$, $F_{j}$, $F_{k}$ of $G$, we have $F_{i}\cap F_j=F_j\cap F_k= F_i\cap F_k$ whenever $F_i\cap F_j\cap F_k\neq \emptyset$. By a \textit{block graph}, we mean a generalized block graph $G$ such that for any two maximal cliques $F_i$ and $F_j$, we have $\vert F_i\cap F_j\vert\leq 1$.
}
\end{definition}

We now go over a few lemmas and theorems that are required in order to prove our results.

\begin{theorem}[{\cite[Theorem 3.2]{arvindgen20}}] \label{thrm:2}
Let $G$ be a finite simple graph and $v$ be a non-free vertex of $G$. Then $$\mathcal{J}_{K_{m},G}=\mathcal{J}_{K_{m},G_{v}} \cap ((x_{iv}: i\in [m])+\mathcal{J}_{K_{m},G-v}).$$ 
\end{theorem}

\begin{lemma}[{\cite[Lemma 3.4]{arvindgen20}}]\label{lem:2} Let $G$ be a graph and $v$ be a non free vertex of $G$. Then $\mathrm{max}\{\mathrm{iv}(G_{v}),\mathrm{iv}(G-v),\mathrm{iv}(G_{v}-v)\} < \mathrm{iv}(G)$, where $\mathrm{iv}(G)$ denotes the number of non-free vertex of $G$.
\end{lemma}

\begin{theorem}[{\cite[Corollary 3.4]{ci20}}] \label{thrm:1} Let $m,n \geq 2$ and let $G$ be a connected block graph on the vertex set $[n]$. Then $\mathrm{depth}(S/\mathcal{J}_{K_{m},G})= m+n-1$.
\end{theorem}

\begin{theorem}[Auslander-Buchbaum Formula]
Let $(A,\mathfrak{m})$ be a Noetherian local ring and $ M \neq 0 $ be a finite A-module. If the projective dimension $\mathrm{pd}_{A}(M)$ of $M$ is finite, then 
$$\mathrm{pd}_{A}(M)+\mathrm{depth}(M)= \mathrm{depth}(A).$$ 
\end{theorem}

Note that in this paper, by saying the depth of the generalized binomial edge ideal of a graph $G$, we mean the depth of the quotient ring $S/\mathcal{J}_{K_m,G}$.

\section{Lower Bound for depth of Generalized Binomial Edge Ideals}\label{sec3}
In this section, we generalize the notion of $d$-compatible map given in \cite{rouzbahani2022depth}. Using this map, we give a combinatorial lower bound of $\mathrm{depth}(S/\mathcal{J}_{K_{m},G}) $.

\begin{definition}\label{def:dcomp}{\rm
Let $\mathcal{G}$ denote the set of all simple graphs and $\mathcal{K}$ denote the set of all complete graphs. Let $ G \in \mathcal{G}$ be a graph on $[n]$ and $K_{m}\in \mathcal{K}$ be the complete graph on $m$ vertices. Now, we are set to define the notion of $d$-compatible maps from $\mathcal{K\times G}$ to $\mathbb{N}_{0}$. A map $\Psi:\mathcal{K \times G}\rightarrow\mathbb{N}_{0}$ is said to be \textit{$d$-compatible} if the following conditions are fulfilled:
\begin{enumerate}
\item If $G= \bigsqcup_{i=1}^{t}K_{n_{i}}$, then $\Psi(K_{m},G)\leq t(m-1)+\sum_{i=1}^{t}n_{i}$;
\item If $G\neq \bigsqcup_{i=1}^{t}K_{n_{i}}$, then there exists a non-free vertex $v \in V(G)$ such that 
\begin{enumerate}
\item $\Psi(K_{m},G-v)\geq \Psi(K_{m},G)$,
\item $\Psi(K_{m},G_{v})\geq \Psi(K_{m},G)$,
\item $\Psi(K_{m},G_{v}-v)\geq \Psi(K_{m},G)-1$.
\end{enumerate}
\end{enumerate}
}
\end{definition}

\begin{theorem}\label{thmd-comp}
 Let $G$ be a graph and $\Psi$ be a $d$-compatible map. Then, $$\mathrm{depth}(S/\mathcal{J}_{K_{m},G})\geq \Psi(K_{m},G).$$ \end{theorem}
\begin{proof} We proceed by using induction on $\mathrm{iv}(G)$, the number of non-free vertices of $G$. If $\mathrm{iv}(G)=0$, then $G$ is a disjoint union of complete graphs, that is, $G=\bigsqcup_{i=1}^{t}K_{n_{i}}$ for some $n_{i}\geq 1$ and $1\leq i\leq t$. Since $\mathcal{J}_{K_{m},K_{n_1}},\ldots, \mathcal{J}_{K_{m},K_{n_t}}$ are generated in pairwise disjoint sets of variables, we have
$$ S/\mathcal{J}_{K_{m},G} \cong S_{1}/\mathcal{J}_{K_{m},K_{n_1}} \otimes_{K} S_{2}/\mathcal{J}_{K_{m},K_{n_2}} \otimes_{K} \cdots \otimes_{K} S_{t}/\mathcal{J}_{K_{m},K_{n_t}},$$ 
where $S_{i}=K[x_{1j},\cdots,x_{mj}:j\in V(K_{n_i})]$ for $1\leq i\leq t$. Therefore, 
\begin{align*}
    \mathrm{depth}(S/\mathcal{J}_{K_{m},G})&= \sum_{i=1}^{t} \mathrm{depth}(S_{i}/\mathcal{J}_{K_{m},K_{n_i}})\\
    &=t(m-1)+\sum_{i=1}^{t}n_{i}\quad (\text{by Theorem \ref{thrm:1}}).
\end{align*}
Hence, the condition (1) in the definition of $d$-compatible map gives $\mathrm{depth}(S/\mathcal{J}_{K_{m},G}) \geq \Psi(K_{m},G)$.
Now, let us assume $\mathrm{iv}(G)>0$. Then $G\neq \bigsqcup_{i=1}^{t}K_{n_i}$, and there exists a non-free vertex $v\in V(G)$ for $\Psi$ which satisfies the condition (2) of Definition \ref{def:dcomp}. Assume that the result holds for all graphs $G'$ with $\mathrm{iv}(G')<\mathrm{iv}(G)$. By Theorem \ref{thrm:2}, we can write
$$\mathcal{J}_{K_{m},G}=\mathcal{J}_{K_{m},G_{v}} \cap ((x_{iv}: i\in [m])+\mathcal{J}_{K_{m},G-v}).$$
Again, observe that 
$$\mathcal{J}_{K_{m},G_{v}}+((x_{iv}: i\in [m])+\mathcal{J}_{K_{m},G-v})= ((x_{iv}: i\in [m])+\mathcal{J}_{K_{m},G_{v}-v}).$$
Therefore, we obtain the following short exact sequence:
$$0 \longrightarrow S/\mathcal{J}_{K_{m},G} \longrightarrow S/\mathcal{J}_{K_{m},G_{v}} \oplus S/((x_{iv}: i\in [m])+\mathcal{J}_{K_{m},G-v})$$ $$  \longrightarrow S/((x_{iv}: i\in [m])+\mathcal{J}_{K_{m},G_{v}-v}) \longrightarrow 0.$$ 
namely, 
$$0 \longrightarrow S/\mathcal{J}_{K_{m},G} \longrightarrow S/\mathcal{J}_{K_{m},G_{v}} \oplus S_{v}/\mathcal{J}_{K_{m},G-v} \longrightarrow S_{v}/\mathcal{J}_{K_{m},G_{v}-v} \longrightarrow 0, $$ 
where $S_{v}= K[x_{ij}:i\in [m],j\in V(G-v)]$. Now, it is well-known that $\mathrm{depth}(M\oplus N)=\mathrm{min}\{\mathrm{depth}(M),\mathrm{depth}(N)\}$, where $M,N$ are finitely generated modules over $S$. Therefore, the well-known Depth Lemma gives
\begin{equation*}
\mathrm{depth}(S/\mathcal{J}_{K_{m},G})\geq \mathrm{min} \{ \mathrm{depth}(S/\mathcal{J}_{K_{m},G_{v}}),\mathrm{depth}(S_{v}/\mathcal{J}_{K_{m},G-v}),
\end{equation*}
\begin{equation}\label{ineq1}
\quad\quad\quad\quad\quad\mathrm{depth}(S_{v}/\mathcal{J}_{K_{m},G_{v}-v})+1\}.
\end{equation}
Again, due to Lemma \ref{lem:2}, Definition \ref{def:dcomp} part (2), and induction hypothesis, we have the following inequalities:
\begin{equation}\label{ineq2}
    \mathrm{depth}(S/\mathcal{J}_{K_{m},G_{v}}) \geq \Psi(K_{m},G_{v})\geq \Psi(K_{m},G),
\end{equation}
\begin{equation}\label{ineq3}
    \mathrm{depth}(S_{v}/\mathcal{J}_{K_{m},G-v}) \geq \Psi(K_{m},G-v)\geq \Psi(K_{m},G),
\end{equation}
\begin{equation}\label{ineq4}
    \mathrm{depth}(S_{v}/\mathcal{J}_{K_{m},G_{v}-v}) \geq \Psi(K_{m},G_{v}-v) \geq \Psi(K_{m},G)-1.
\end{equation}

\noindent Hence, the inequalities \eqref{ineq1}, \eqref{ineq2}, \eqref{ineq3}, and \eqref{ineq4} together gives $\mathrm{depth} (S/\mathcal{J}_{K_{m},G})\geq \Psi(K_{m},G)$.
\end{proof}

\begin{definition}{\rm
Let $ G $ be a connected graph and $u,v \in V(G) $ be any two vertices. Then the \textit{distance} between the vertices $ u $ and $ v $ in $ G $, denoted by $ d_{G}(u,v) $, is the length of a shortest path between $u$ and $v$ in $G$. Now, the \textit{diameter} of $G$, denoted by $\mathrm{diam}(G)$, is defined as 
$$\mathrm{diam}(G):= \mathrm{max}\{d_{G}(u,v): u,v \in V(G)\}.$$
Note that $d_{G}(u,u)=0$ for any $u\in V(G)$ and thus, if $G$ is an empty graph (i.e., $E(G)=\emptyset$), then $\mathrm{diam}(G)=0.$
}
\end{definition}

Let G be a graph with the connected components $ G_{1} $,...,$ G_{t} $. Now, we set $d(G)= i(G)+ \sum_{i=1}^{t}\mathrm{diam}(G_{i}) $, where $i(G)$ denotes the number of isolated vertices of $G$. Also, we denote the number of free vertices of $G$ by $f(G)$. In the following theorem, we show the existence of a $d$-compatible map and give a combinatorial lower bound of the depth of $S/\mathcal{J}_{K_m,G}$ using $f(G)$, $d(G)$ and the number of connected components $t$ of $G$.

\begin{theorem}\label{thmlb}
The map $\Psi: \mathcal{K} \times \mathcal{G} \rightarrow \mathbb{N}_{0} $ defined by $$\Psi(K_{m},G)= (m-2)t+f(G)+d(G)$$ is $d$-compatible. Hence, $\mathrm{depth}(S/J_{K_{m},G}) \geq (m-2)t+f(G)+d(G)$.
\end{theorem}

\begin{proof}
Let $G$ be a simple graph on $n$ vertices with $t$ connected components. Fix a complete graph $K_{m}\in\mathcal{K}$. If $G=\bigsqcup_{i=1}^{t}K_{n_{i}}$ with $n_{i}\geq 1$ for all $1\leq i\leq t$, then clearly we have $f(G)=\sum_{i=1}^{t} n_{i}$ and $d(G)=t$. Hence,
\begin{align*}
 \Psi(K_{m},G) &= (m-2)t+f(G)+d(G)\\
  & = (m-2)t+\sum_{i=1}^{t} n_{i}+t\\
  & = (m-1)t+\sum_{i=1}^{t} n_{i}.
\end{align*}
 Now, let us assume $G\neq \bigsqcup_{i=1}^{t}K_{n_{i}}$. Then, it has been shown in the proof of \cite[Theorem 3.4]{rouzbahani2022depth} that there exists a non-free vertex $v$ of $G$ such that
 \begin{equation}\label{eq5}
     f(G-v)+d(G-v) \geq f(G)+ d(G).
 \end{equation}
 Let $t'$ denote the number of connected components of $G-v$. Clearly,  $t' \geq t$ and 
 \begin{equation}\label{eq6}
     (m-2)t' \geq (m-2)t.
 \end{equation}
Therefore, from \eqref{eq5} and \eqref{eq6}, we have
\begin{align*}
 \Psi(K_{m},G-v) &=(m-2)t'+f(G-v)+d(G-v)\\
  & \geq(m-2)t+f(G)+d(G)\\
  & = \Psi(K_{m}, G).
\end{align*}
Again, from the proof of \cite[Theorem 3.4]{rouzbahani2022depth} it follows that for any non-free vertex $v$ of $G$, $d(G_v)\geq d(G)-1$ and $d(G_v-v)\geq d(G)-1$. Since $\mathrm{iv}(G)+f(G)=n$, we have $f(G_v)\geq f(G)+1$ and $f(G_v-v)\geq f(G)$ by Lemma \ref{lem:2}. Therefore,
\begin{equation}\label{eq7}
    f(G_{v})+d(G_{v}) \geq f(G)+ d(G),       
\end{equation}
\begin{equation}\label{eq8}
    f(G_{v}-v)+d(G_{v}-v) \geq f(G)+ d(G)-1.
\end{equation}
 Now, the number of connected components of $G_{v}$ is also $t$, and thus, from inequality \eqref{eq7} the following hold:
 \begin{align*}
 \Psi(K_{m},G_{v}) & =(m-2)t+f(G_{v})+d(G_{v})\\
  & \geq (m-2)t+f(G)+d(G)\\
  & = \Psi(K_{m}, G).
 \end{align*}
 Also, $v$ being a non-free vertex of $G$, the number of connected components of $G_v-v$ is also $t$, and so, inequality \eqref{eq8} gives
 \begin{align*}
     \Psi(K_{m},G_{v}-v) & =(m-2)t+f(G_{v}-v)+d(G_{v}-v)\\
  & \geq (m-2)t+f(G)+d(G)-1\\
  & = \Psi(K_{m}, G)-1.
 \end{align*}
  Therefore, the given map $\Psi$ satisfies all the conditions given in Definition \ref{def:dcomp}, and thus, $\Psi(K_{m},G)= (m-2)t+f(G)+d(G)$ is a $d$-compatible map. Hence, by Theorem \ref{thmupperbound}, $\mathrm{depth}(S/J_{K_{m},G}) \geq (m-2)t+f(G)+d(G)$.
\end{proof}

\medskip

\section{Upper Bound for depth of Generalized Binomial Edge Ideals}\label{sec4}
In this section, we give a combinatorial upper bound of $ \mathrm{depth}(S/\mathcal{J}_{K_{m},G})$ in terms of the vertex-connectivity of the graph $G$. Since the depth of the generalized binomial edge ideal of a complete graph is well-known and for disconnected graphs, the depth will be the sum of the depth of connected components, it is enough to consider $G$ is a connected non-complete graph.

\begin{definition} {\rm
Let $A$ be a polynomial ring and $I$ be an ideal of $A$. The cohomological dimension of $I$, denoted by $\mathrm{cd}(I)$, is defined as follows: 
$$ \mathrm{cd}(I):= \mathrm{max}\{ i \in \mathbb{N} : H_{I}^{i}(A)\neq 0 \},$$ 
where $H_{I}^{i}(M)$ denotes the $i$'th local cohomology module of an $A$-module $M$ with support in $I$.
}
\end{definition}

\begin{definition}{\rm
    The \textit{vertex-connectivity} of a connected graph $G$ is the minimum cardinality of a set of vertices $S\subseteq V(G)$ such that $G-S$ is disconnected. We denote the vertex-connectivity of $G$ by $\kappa(G)$. Since no set of vertices removed from a complete graph makes it disconnected, as a convention, it is assigned that $\kappa(K_n)=n-1$.
    }
\end{definition}

\begin{theorem}
[{\cite[Theorem 2.11]{katsabekis2022arithmetical}}]\label{thrm:3} Let $ G $ be a connected graph on the vertex set $ [n] $ and $ K $ be a field of any characteristic. If $ G $ is not the complete graph, then $$ \mathrm{cd}(\mathcal{J}_{K_{m},G}) \geq mn-m-n+\kappa(G),$$ 
where $\kappa(G)$ denotes the vertex-connectivity of $ G $.
\end{theorem}

\begin{proposition}[{\cite[Proposition 4.1]{peskine1973dimension}}]\label{propcdpd}
If $A$ is a polynomial ring over a field of characteristic $p>0$, then $ \mathrm{cd}(I)\leq \mathrm{pd}_{A}(A/I) $.
\end{proposition}

\begin{theorem}\label{thmupperbound}
Let $G$ be a non-complete connected graph on $n$ vertices. If the vertex-connectivity of $G$ is $\kappa(G)$, then 
$$\mathrm{depth} (S/\mathcal{J}_{K_{m},G})\leq m+n-\kappa(G).$$
\end{theorem}

\begin{proof}
First, let us assume $\mathrm{char}(K)=p>0$. Then, by Theorem \ref{thrm:3} and Proposition\ref{propcdpd}, we have
$$ \mathrm{pd}(S/\mathcal{J}_{K_{m},G})\geq \mathrm{cd}( \mathcal{J}_{K_{m},G}) \geq mn-m-n+\kappa(G).$$
Therefore, by Auslander-Buchsbaum Formula,
\begin{align*}
\mathrm{depth}(S/\mathcal{J}_{K_{m},G}) & = \mathrm{depth}(S) - \mathrm{pd}(S/\mathcal{J}_{K_{m},G}) \\
& \leq mn-mn+m+n-\kappa(G) \\
& = m+n-\kappa(G).
\end{align*}
Now, let $K$ be any field with $\mathrm{char}(K)=0$. Since projective dimension is preserved under faithfully flat extension, so is depth by Auslander-Buchsbaum Formula. Therefore, without loss of generality, it is enough to assume $K=\mathbb{Q}$. Let $R=\mathbb{Z}[x_{i,j}: i\in[m]\,\,\text{and}\,\,j\in[n]]$ and $J_{K_m,G}$ be the ideal in $R$ generated by all binomial of the form $x_{ik}x_{jl}-x_{il}x_{jk}$, where $\{i,j\}\in E(K_m)$ with $i<j$ and $\{k,l\}\in E(G)$ with $k<l$. Then, due to \cite[Theorem 2.3.5]{hhunpublished}, we get for $p>>0$ 
$$\mathrm{pd}_{S}(S/\mathcal{J}_{Km,G})=\mathrm{pd}_{R\otimes_{\mathbb{Z}}\mathbb{Q}}(R\otimes_{\mathbb{Z}} \mathbb{Q}/J_{K_m,G}\otimes_{\mathbb{Z}}\mathbb{Q})=\mathrm{pd}_{R\otimes_{\mathbb{Z}}\mathbb{F}_{p}}(R/J_{K_m,G}\otimes_{\mathbb{Z}}\mathbb{F}_{p}).$$
Hence, from the proof of characteristic $p>0$ case and Auslander-Buchsbaum Formula, we get the desired result. 
\end{proof}

\section{Examples}\label{sec5}
In this section, first we provide a family of graphs for which the lower bound and upper bound of the depth of their generalized binomial edge ideals are equal. We calculate the depth of generalized binomial edge ideals of cyclic graphs. As a remark, we get an infinite class of graphs for which the upper bound is tight and arbitrarily larger than the lower bound. Also, we give an infinite class of graphs for which the lower bound is tight and arbitrarily smaller than the upper bound. Finally, we discuss the depth of generalized binomial edge ideals of those graphs whose binomial edge ideals are Cohen-Macaulay.
\medskip

Let $\mathcal{H}_1$ be the class of non-complete connected graphs $G$, for which there exists a complete graph $K_s$ such that any two maximal cliques of $G$ intersect at $K_s$. A graph in this class is illustrated in \Cref{figh1}. 

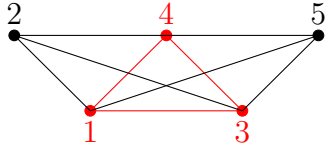
\begin{figure}[h]

\begin{tikzpicture}
\filldraw[red] (0,0) circle (2pt)node[anchor=north]{$1$};
\filldraw[red] (2,0) circle (2pt)node[anchor=north]{$3$};
\filldraw[red] (1,1) circle (2pt)node[anchor=south]{$4$};
\filldraw[black] (-1,1) circle (2pt)node[anchor=south]{$2$};
\filldraw[black] (3,1) circle (2pt)node[anchor=south]{$5$};

\draw[red] (0,0) -- (2,0) -- (1,1) -- cycle;
\draw[black] (0,0) -- (-1,1);
\draw[black] (2,0) -- (-1,1);
\draw[black] (1,1) -- (-1,1);
\draw[black] (2,0) -- (3,1);
\draw[black] (0,0) -- (3,1);
\draw[black] (1,1) -- (3,1);

\end{tikzpicture}
\caption{A graph $G\in \mathcal{H}_1$ with $K_s=K_3=G- \{2,5\}$.}
\label{figh1}
\end{figure}

\begin{theorem}\label{thmub=lb}
    Let $G\in \mathcal{H}_1$ be a graph on $n$ vertices. Then
    $$\mathrm{depth}(S/\mathcal{J}_{K_m,G})=(m-2)+f(G)+d(G)=m+n-\kappa(G).$$
\end{theorem}

\begin{proof}
Let any two maximal cliques of $G$ intersect at a complete graph $K_s$ on $s$ vertices. Then, to disconnect the graph $G$, one has to remove at least the vertices of the complete graph where any two maximal cliques intersect. Since $G$ is non-complete, removal of the vertices of $K_s$ is enough to disconnect $G$. Thus, the vertex-connectivity of $G$ is $\kappa(G)=s$. Also, note that the only non-free vertices of $G$ are the vertices of $K_s$. Thus, $f(G)=n-s$. Again, it is easy to verify that $d_{G}(u,v)=0$ whenever $u,v$ belongs to the same maximal clique and $d_{G}(u,v)=2$ whenever $u,v$ belong to distinct maximal cliques. Therefore, $d(G)=2$. Since the graph $G$ is connected, $t=1$. By Theorems \ref{thmlb} and \ref{thmupperbound}, 
\begin{align*}
\mathrm{depth}(S/\mathcal{J}_{K_m,G}) & \geq (m-2)t+f(G)+d(G)\\
&= m-2+n-s+2\\
&= m+n-s\\
&= m+n-\kappa(G)\\
&\geq \mathrm{depth}(S/\mathcal{J}_{K_m,G}).
\end{align*}
Hence, the assertion follows.
\end{proof}

\begin{definition}[{\cite[Definition 3.4]{ass23}}]\label{defcomp} {\rm
    Let $G$ be a simple graph. For a set $V=\{v_{1},\ldots,v_{k}\}\subseteq V(G)$, we write $G_{V}=G_{v_{1} v_{2}\cdots v_{k}}:=(\ldots ((G_{v_{1}})_{v_{2}})\ldots)_{v_{k}}$. Due to \cite[Proposition 3.2]{ass23}, $G_{V}$ does not depend on the ordering of the elements of $V$, and thus, the definition of $G_{V}$ is well-defined. A set $W\subseteq V(G)$ is said to be a \textit{completion set} of $G$ if $G_{W}$ is a disjoint union of complete graphs.
    }
\end{definition}

Let $\mathcal{H}_2$ be the class of those graphs $G$ such that either $G$ is a cycle $C_n$ or $G=(C_n)_W$ for some set of vertices $W\subseteq V(C_n)$ (see \Cref{figh2}).

\begin{figure}[h]
\begin{tikzpicture}

\filldraw[black] (0,0) circle (2pt)node[anchor=north]{$1$};
\filldraw[black] (2,0) circle (2pt)node[anchor=north]{$2$};
\filldraw[black] (-1,1.5) circle (2pt)node[anchor=east]{$6$};
\filldraw[black] (3,1.5) circle (2pt)node[anchor=west]{$3$};
\filldraw[black] (0,3) circle (2pt)node[anchor=south]{$5$};
\filldraw[black] (2,3) circle (2pt)node[anchor=south]{$4$};

\draw[black] (0,0) -- (-1,1.5) -- (0,3) -- (2,3) -- (3,1.5) -- (2,0) --(0,0);
\draw[black] (2,0) -- (-1,1.5);
\draw[black] (0,0) -- (3,1.5);
\draw[black] (3,1.5) -- (-1,1.5);
\draw[black] (0,3) -- (3,1.5);

\end{tikzpicture}
\caption{The graph $G\in\mathcal{H}_{2}$ such that $G=(((C_6)_{1})_{2})_{4}$.}
\label{figh2}
\end{figure}
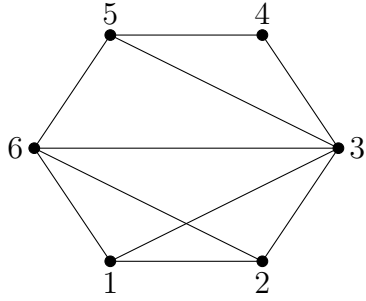

\begin{theorem}\label{thmdepthcycle}
Let $G \in \mathcal{H}_2$ be a non-complete graph. Then $\kappa(G)=2$, and $$\mathrm{depth}(S/\mathcal{J}_{K_{m},G})= m+n-2.$$ 
In particular, for any cycle $C_n$ of length $n\geq 4$, $\mathrm{depth}(S/\mathcal{J}_{K_{m},C_n})= m+n-2$.
\end{theorem}

\begin{proof}
First, let us prove $\kappa(G)=2$ for all non-complete $G\in\mathcal{H}_2$. Let $V(G)=[n]$. Since $G$ is non-complete, there exists a non-free vertex $v\in V(G)$. Without loss of generality, we assume $v=1$. If $G$ is a cycle, then we are done. Suppose $G=(C_n)_W$ for some $W\subseteq [n]$. Then $1\not\in W$. If $2,n\not\in W$, then $G-\{2,n\}$ is a disconnected graph and hence, $\kappa(G)=2$. If one of $2$ and $n$ belong to $W$, then note that there exists a non-free vertex $i\in\mathcal{N}_{G}(1)\setminus \{2,n\}$, otherwise $G$ would be complete. Now, from the structure of graphs in $\mathcal{H}_2$, one can check that removal of $\{1,i\}$ from $G$ disconnect $G$. Thus, $\kappa(G)=2$.\par 
Now, we prove the formula of depth by induction on the number of non-free vertices $\mathrm{iv}(G)$ of $G$. Since $G$ is non-complete, $\mathrm{iv}(G)>0$. Suppose $\mathrm{iv}(G)=1$. Then $f(G)=n-1$ and thus, $G\simeq (C_{n})_{1\ldots n-1}$. But, from Definition \ref{defcomp}, it is easy to verify that any consecutive $n-2$ vertices of $C_n$ form a completion set of $C_n$. Thus, $G$ should be a complete graph, which is a contradiction to the fact that $\mathrm{iv}(G)=1$. So, the base case is when $\mathrm{iv}(G)=2$. In this case, we have $f(G)=n-2$. Since $G$ is non-complete, $d(G)\geq 2$. Suppose $d(G)>2$. Now, $G$ is connected and $d(G)>2$ imply there exist $u_1,u_2\in V(G)$ such that $d_{G}(u_1,u_2)>2$. Since $G\in\mathcal{H}_2$, there will be two distinct induced paths between $u_1$ and $u_2$, which do not share any common vertex other than $u_1$ and $u_2$. The vertices other than $u_1,u_2$ belonging to any induced path between $u_1$ and $u_2$ are non-free. Thus, $\mathrm{iv}(G)>2$, and this gives a contradiction. Therefore, we have $\mathrm{d}(G)=2$. Also, note that for any non-complete graph $G\in\mathcal{H}_2$, we have $\kappa(G)=2$. Hence, by Theorems \ref{thmlb} and \ref{thmupperbound}, 
\begin{align*}
\mathrm{depth}(S/\mathcal{J}_{K_m,G}) &\geq (m-2)t+f(G)+d(G)\\
&= (m-2)+(n-2)+2\\
&= m+n-2\\
&= m+n-\kappa(G)\\
&\geq \mathrm{depth}(S/\mathcal{J}_{K_m,G}).
\end{align*}
Therefore, $\mathrm{depth}(S/\mathcal{J}_{K_{m},G})=m+n-2$. Now, let us assume that the result is true for all non-complete graphs $G' \in \mathcal{H}_2$ with $\mathrm{iv}(G')<k$. Let $G \in \mathcal{H}_2$ be a graph with $\mathrm{iv}(G)=k>0$. Let $v$ be a non-free vertex of $G$. Then, by Theorem \ref{thrm:2},
$$\mathcal{J}_{K_{m},G}=\mathcal{J}_{K_{m},G_{v}} \cap ((x_{iv}: i\in [m])+\mathcal{J}_{K_{m},G-v}).$$
Hence, we have the following exact sequence:
$$0 \longrightarrow S/\mathcal{J}_{K_{m},G} \longrightarrow S/\mathcal{J}_{K_{m},G_{v}} \oplus S_{v}/\mathcal{J}_{K_{m},G-v} \longrightarrow S_{v}/\mathcal{J}_{K_{m},G_{v}-v} \longrightarrow 0, $$ 
where $S_{v}= K[x_{ij}:i\in [m],j\in V(G-v)]$. Now, by Depth Lemma, 
\begin{equation*}
\mathrm{depth}(S/\mathcal{J}_{K_{m},G})\geq \mathrm{min} \{ \mathrm{depth}(S/\mathcal{J}_{K_{m},G_{v}}),\mathrm{depth}(S_{v}/\mathcal{J}_{K_{m},G-v}),
\end{equation*}
\begin{equation}\label{ineq5.1}
\quad\quad\quad\quad\quad\mathrm{depth}(S_{v}/\mathcal{J}_{K_{m},G_{v}-v})+1\}.
\end{equation}
Note that $G-v$ is a connected block graph such that no block contains more than two cut vertices. Then by Theorem \ref{thrm:1},
\begin{equation}\label{ineq5.2}
    \mathrm{depth}(S_v/\mathcal{J}_{K_m, G-v})= m+(n-1)-1=m+n-2.
\end{equation}

Now consider the graph $G_{v}$. If $G_{v}$ is a complete graph, then $\mathrm{depth}(S/\mathcal{J}_{K_{m},G_{v}})=m+n-1$ by Theorem \ref{thrm:1}. If $G_v$ is not a complete graph, then also $G_v\in\mathcal{H}_{2}$ by the choice of $\mathcal{H}_{2}$. Then due to Lemma \ref{lem:2}, we use induction hypothesis to get $\mathrm{depth}(S/\mathcal{J}_{K_{m},G_{v}})=m+n-2$. Thus, in any situation, we get
\begin{equation}\label{ineq5.3}
    \mathrm{depth}(S/\mathcal{J}_{K_m, G_v})\geq m+n-2.
\end{equation}
Similarly, if $G_{v}-v$ is complete, then $\mathrm{depth}(S/\mathcal{J}_{K_{m},G_{v}-v})=m+n-2$, and if $G_{v}-v$ is non-complete, then $\mathrm{depth}(S/\mathcal{J}_{K_{m},G_{v}-v})=m+n-3$.
Therefore,
\begin{equation}\label{ineq5.4}
    \mathrm{depth}(S/\mathcal{J}_{K_m, G_{v}-v})+1\geq m+n-2.
\end{equation}
From inequalities \eqref{ineq5.1}, \eqref{ineq5.2}, \eqref{ineq5.3}, and \eqref{ineq5.4}, it follows that
$$\mathrm{depth}(S/\mathcal{J}_{K_m, G})\geq m+n-2.$$
Again, $\kappa(G)=2$ and Theorem \ref{thmupperbound} together imply
$$\mathrm{depth}(S/\mathcal{J}_{K_m, G})\leq m+n-2.$$
Hence, the result follows.
\end{proof}

\begin{remark}\label{remub=lb+k}{\rm
    For every $m\geq 2$ and $k\in\mathbb{N}$, let us choose the graph $C_{2k}$. By Theorem \ref{thmdepthcycle}, $\mathrm{depth}(S/\mathcal{J}_{K_m,C_{2k}})=m+2k-2=m+n-\kappa(C_{2k})$. On the other hand, $(m-2)t+f(C_{2k})+d(C_{2k})=m+k-2$. Hence, for each $m\geq 2$ and $k\in\mathbb{N}$, there exists a graph $G$ such that $\mathrm{depth}(S/\mathcal{J}_{K_m,G})$ is equal the upper bound given in Theorem \ref{thmupperbound} and $k$ more than the lower bound given in Theorem \ref{thmlb}.
    }
\end{remark}

Let $\mathcal{H}_3$ be the collection of those generalized block graphs $G$ such that 
\begin{enumerate}[(a)]
    \item for any three distinct maximal cliques $F_i$, $F_j$, $F_k$ of $G$, we have $F_i\cap F_j\cap F_k=\emptyset$,
    \item all the maximal cliques of $G$ are $K_3$,
    \item we can order the maximal cliques as $F_1\cap F_2\simeq K_2$, $F_2\cap F_3\simeq K_1$, $F_3\cap F_4\simeq K_2$, and goes on.
\end{enumerate}
A graph belonging to $\mathcal{H}_3$ has been shown in \Cref{figh3}.

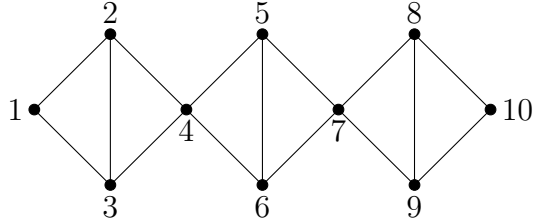
\begin{figure}[h]
\begin{tikzpicture}

\filldraw[black] (0,0) circle (2pt)node[anchor=east]{$1$};
\filldraw[black] (2,0) circle (2pt)node[anchor=north]{$4$};
\filldraw[black] (4,0) circle (2pt)node[anchor=north]{$7$};
\filldraw[black] (6,0) circle (2pt)node[anchor=west]{$10$};
\filldraw[black] (1,1) circle (2pt)node[anchor=south]{$2$};
\filldraw[black] (3,1) circle (2pt)node[anchor=south]{$5$};
\filldraw[black] (5,1) circle (2pt)node[anchor=south]{$8$};
\filldraw[black] (1,-1) circle (2pt)node[anchor=north]{$3$};
\filldraw[black] (3,-1) circle (2pt)node[anchor=north]{$6$};
\filldraw[black] (5,-1) circle (2pt)node[anchor=north]{$9$};

\draw[black] (0,0) -- (1,-1);
\draw[black] (0,0) -- (1,1);
\draw[black] (1,1) -- (1,-1);
\draw[black] (1,1) -- (2,0);
\draw[black] (1,-1) -- (2,0);
\draw[black] (2,0) -- (3,1);
\draw[black] (2,0) -- (3,-1);
\draw[black] (3,1) -- (3,-1);
\draw[black] (3,1) -- (4,0);
\draw[black] (3,-1) -- (4,0);
\draw[black] (4,0) -- (5,-1);
\draw[black] (4,0) -- (5,1);
\draw[black] (5,1) -- (5,-1);
\draw[black] (5,-1) -- (6,0);
\draw[black] (5,1) -- (6,0);
\end{tikzpicture}
\caption{The graph $G_3\in\mathcal{H}_{3}$ with $6$ maximal cliques.}
\label{figh3}
\end{figure}

\begin{figure}[h]
\begin{tikzpicture}

\filldraw[black] (0,0) circle (2pt)node[anchor=east]{$1$};
\filldraw[black] (2,0) circle (2pt)node[anchor=north]{$4$};
\filldraw[black] (1,1) circle (2pt)node[anchor=south]{$2$};
\filldraw[black] (1,-1) circle (2pt)node[anchor=north]{$3$};
\filldraw[black] (4,0) circle (2pt)node[anchor=north]{$5$};

\draw[black] (0,0) -- (1,-1);
\draw[black] (0,0) -- (1,1);
\draw[black] (1,1) -- (1,-1);
\draw[black] (1,1) -- (2,0);
\draw[black] (1,-1) -- (2,0);
\draw[black] (2,0) -- (4,0);
\end{tikzpicture}
\caption{The graph $G_1$ considered in the proof of Theorem \ref{thmd=lb=ub-k}.}
\label{figh4}
\end{figure}
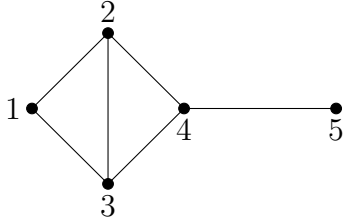

\begin{theorem}\label{thmd=lb=ub-k}
  For every $m\geq 2$ and $k\in\mathbb{N}$, there exists a connected graph $G_k$ such that 
  $$\mathrm{depth}(S/\mathcal{J}_{K_m,G_k})=(m-2)+f(G_k)+d(G_k)=m+n-\kappa(G_k)-k.$$  
\end{theorem}

\begin{proof}
For $k=1$, we choose the graph $G_1$ given in \Cref{figh4}. Then it is clear from the structure of $G_1$ that $n=5$, $f(G_1)=2$, $d(G_1)=3$, $\kappa(G_1)=1$. Note that $G_1$ is a generalized block graph, and thus, $\mathrm{depth}(S/\mathcal{J}_{K_m,G_1})=m+3$ by \cite[Theorem 3.3]{ci20}. In this case, $(m-2)+f(G_1)+d(G_1)=m+n-\kappa(G_1)-1=m+3$. Now, for $k\geq 2$, let us choose the graph $G_k\in\mathcal{H}_3$ such that $G_k$ has $2k$ maximal cliques. Then, the number of minimal cutsets of $G_k$ with cardinality $2$ is $k$. Note that the number of vertices of $G_k$ is $n=4k-(k-1)=3k+1$. Therefore, by \cite[Theorem 3.3]{ci20}, we get $\mathrm{depth}(S/\mathcal{J}_{K_m,G_k})=n+(m-1)-(2-1)k=2k+m$. Again, from the structure of $G_k$, we observe that the only free vertices of $G_k$ are $v_1$ and $v_2$, where $F_1\setminus F_2=\{v_1\}$ and $F_{2k}\setminus F_{2k-1}=\{v_2\}$, i.e. $f(G_k)=2$. Again, it is clear that $\mathrm{diam}(G_k)=d_{G_k}(v_1,v_2)=2k$, i.e. $d(G_k)=2k$. Thus, we have $(m-2)+f(G_k)+d(G_k)=2k+m=\mathrm{depth}(S/\mathcal{J}_{K_m,G_k})$. Now, due to $k>1$, we have $\kappa(G_k)=1$. Thus, $m+n-\kappa(G_k)-k=2k+m=\mathrm{depth}(S/\mathcal{J}_{K_m,G_k})$.
\end{proof}

To give a combinatorial description of Cohen-Macaulay binomial edge ideals, Bolognini et al. \cite{acc} introduced two combinatorial notions of graphs exploiting the cutsets of $G$ and unmixed property of $J_G$ (which is purely combinatorial): accessible graphs and strongly unmixed binomial edge ideals. Since the strongly unmixed property of binomial edge ideals is purely combinatorial, we will call strongly unmixed graphs instead of saying strongly unmixed binomial edge ideals. Let $G$ be a simple graph with $c$ connected components. Then $J_G$ is unmixed if and only if $c_{G}(T)=\vert T\vert +c$ for every $T\in\mathcal{C}(G)$. A non-empty cutset $T$ of $G$ is said to be \textit{accessible} if there exists $v\in T$ such that $T\setminus \{v\}\in \mathcal{C}(G)$. A graph $G$ is said to be \textit{accessible} if $J_G$ is unmixed and every non-empty cutset of $G$ is accessible. A graph $G$ is said to be \textit{strongly unmixed} if every connected component of $G$ is complete or if $J_G$ is unmixed and there exists a cut vertex $v$ of $G$ for which $G-v$, $G_v$, $G_{v}-v$ are strongly unmixed. The following has been proved in \cite{acc}:
$$G\text{ is strongly unmixed} \implies J_G\text{ is Cohen-Macaulay}\implies G\text{ is accessible}.$$
The authors conjectured on the converse of the above implications. The conjecture is proved for several classes: chordal and bipartite graphs \cite{acc}; chain of cycles with whiskers \cite{lmrr23}; and $r$-regular $r$-connected blocks with whiskers \cite{sswhisker22}. However, when $m\geq 3$, $\mathcal{J}_{K_m,G}$ is Cohen-Macaulay if and only if $G$ is a disjoint union of complete graphs (by \cite[Corollary 4.3]{acr23}. Thus, if $G$ is a strongly unmixed non-complete graph, then $\mathcal{J}_{K_m,G}$ is not Cohen-Macaulay for $m\geq 3$. But, we can still calculate the depth of generalized binomial edge ideals of strongly unmixed graphs, and that is independent of the field (see Theorem \ref{thmgensu}). If the above mentioned conjecture is true, then the strongly unmixed graphs are nothing but the class of accessible graphs.

\begin{theorem}\label{thmgensu}
    Let $G$ be a strongly unmixed connected graph on the vertex set $[n]$. Then $\mathrm{depth}(\mathcal{S}/\mathcal{J}_{K_m,G})=m+n-1$. Moreover, if $G$ is a strongly unmixed graph on $n$ vertices with $t$ connected components, then $\mathrm{depth}(\mathcal{S}/\mathcal{J}_{K_m,G})=(m-1)t+n$.
\end{theorem}

\begin{proof}
    It is enough to consider that $G$ is connected. If $G$ is a complete graph, then the result holds from Theorem \ref{thrm:1}. Thus, let us assume $G$ is a non-complete graph. Then, by the definition of strongly unmixedness, there exists a cut vertex $v$ of $G$ such that $G-v$, $G_v$, and $G_{v}-v$ are strongly unmixed. Now, by Lemma \ref{lem:2}, $\mathrm{iv}(G-v)<\mathrm{iv}(G)$. Since $v$ is a cut vertex of $G$ and $J_G$ is unmixed, $G-v$ has exactly two connected components. Therefore, using the induction hypothesis, we get $\mathrm{depth}(S_v/\mathcal{J}_{K_m,G-v})=2(m-1)+n-1$. Also, by Lemma \ref{lem:2}, we have $\mathrm{iv}(G_v),\mathrm{iv}(G_v-v)<\mathrm{iv}(G)$. Since $G$ is connected, $G_v$ and $G_v-v$ both are connected. Hence, by induction hypothesis, we get $\mathrm{depth}(S/\mathcal{J}_{K_m,G_v})=m+n-1$ and $\mathrm{depth}(S_v/\mathcal{J}_{K_m,G_v-v})=m+n-2$. Thus, by Depth Lemma,
\begin{align*}
 \mathrm{depth}(S/\mathcal{J}_{K_{m},G}) &\geq \mathrm{min}\{
 \mathrm{depth}(S_{v}/\mathcal{J}_{K_{m},G-v}), \mathrm{depth}(S/\mathcal{J}_{K_{m},G_{v}}),\\ &\hspace{1.45cm}\mathrm{depth}(S_{v}/\mathcal{J}_{K_{m},G_{v}-v})+1\}\\
 &=\{2(m-1)+n-1,m+n-1,m+n-2+1\}\\
 &=m+n-1.
\end{align*}
Hence, it follows from Theorem \ref{thmupperbound} that $\mathrm{depth}(S/J_{K_m,G})=m+n-1$.
\end{proof}

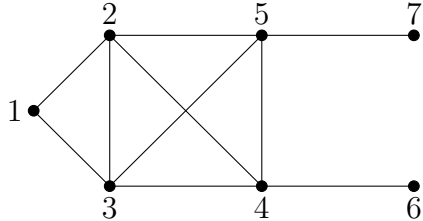
\begin{figure}[h]
\begin{tikzpicture}

\filldraw[black] (0,0) circle (2pt)node[anchor=east]{$1$};
\filldraw[black] (1,1) circle (2pt)node[anchor=south]{$2$};
\filldraw[black] (1,-1) circle (2pt)node[anchor=north]{$3$};
\filldraw[black] (3,1) circle (2pt)node[anchor=south]{$5$};
\filldraw[black] (3,-1) circle (2pt)node[anchor=north]{$4$};
\filldraw[black] (5,1) circle (2pt)node[anchor=south]{$7$};
\filldraw[black] (5,-1) circle (2pt)node[anchor=north]{$6$};

\draw[black] (0,0) -- (1,1);
\draw[black] (0,0) -- (1,-1);
\draw[black] (1,1) -- (1,-1);
\draw[black] (1,1) -- (3,1);
\draw[black] (1,1) -- (3,-1);
\draw[black] (1,-1) -- (3,1);
\draw[black] (1,-1) -- (3,-1);
\draw[black] (3,1) -- (3,-1);
\draw[black] (3,1) -- (5,1);
\draw[black] (3,-1) -- (5,-1);

\end{tikzpicture}
\caption{A graph $G$ with $(m-2)t+f(G)+d(G)<\mathrm{depth}(S/\mathcal{J}_{K_m,G})<m+n-\kappa(G)$.}

\label{figh5}
\end{figure}

\begin{example}\label{exlb<d<ub}{\rm
    Till now, we have discussed several classes of graphs $G$ for which either $\mathrm{depth}(S/\mathcal{J}_{K_m,G})$ is equal to the lower bound given in Theorem \ref{thmlb} or the upper bound given in Theorem \ref{thmupperbound}. However, there exists some graph $G$ for which 
    $$(m-2)t+f(G)+d(G)<\mathrm{depth}(S/\mathcal{J}_{K_m,G})<m+n-\kappa(G).$$
    For instance, consider the graph shown in \Cref{figh5}. Then, $G$ is a generalized binomial edge ideal and using the formula given in \cite[Theorem 3.3]{ci20}, we get $\mathrm{depth}(S/\mathcal{J}_{K_m,G})=m+5$. Whereas, $(m-2)t+f(G)+d(G)=m+4$ and $m+n-\kappa(G)=m+6$.
    }
\end{example}

\section*{Acknowledgement}
Kamalesh Saha would like to thank the National Board for Higher Mathematics (India) for the financial support through the NBHM Postdoctoral Fellowship. Also, Kamalesh Saha was partially supported by an Infosys Foundation fellowship.

\printbibliography

\end{document}